\documentclass [a4paper,12pt]{amsart}
\usepackage{amsthm}
\usepackage{amsmath}
\usepackage{amssymb}
\usepackage[ansinew]{inputenc}
\usepackage{amscd}
\usepackage[ansinew]{inputenc}
\usepackage[mathscr]{euscript}
\usepackage{color}

\newtheorem{thm}{Theorem}
\newtheorem{cor}[thm]{Corollary}
\newtheorem{lem}[thm]{Lemma}
\newtheorem{prop}[thm]{Proposition}

\theoremstyle{definition}
\newtheorem{ex}[thm]{Example}

\theoremstyle{definition}
\newtheorem{defn}[thm]{Definition}

\theoremstyle{definition}
\newtheorem{rem}[thm]{Remark}

\theoremstyle{definition}

\def\Q{\mathbb Q}
\def\C{\mathbb C}
\def\R{\mathbb R}

\def\Z{\mathbb Z}

\def\dim{\operatorname{dim}}

\def\supp {\mathrm{supp}}

\def\ini{\operatorname{in}}

\def\ord {\mathrm{ord}}

\def\O{\mathcal O}

\def\m{\mathbf m}
\def\bfu{\mathbf u}
\def\bfv{\mathbf v}

\def\e{\mathbf e}
\def\l0{\lambda_0}

\def\LL{\textnormal{\texttt{L}}}

\def\geq{\geqslant}
\def\leq{\leqslant}
\def\lct{\operatorname{lct}}
\def\LM{\operatorname{LM}}

\def\LM{\operatorname{\textsc{lm}}}
\def\Mon{\operatorname{Mon}}

\def\DP{\operatorname{DP}}
\def\*{\color{red}\blacksquare\hspace{-0.1cm}\blacksquare\hspace{-0.1cm}\blacksquare}

\subjclass[$2010$ Mathematics Subject Classification]{Primary
13H15; Secondary 32S05, 14B05}

\begin{document}
\author{Carles Bivi\`a-Ausina}
\title[Log canonical threshold and diagonal ideals]{Log canonical threshold and diagonal ideals}

\address{Institut Universitari de Matemàtica Pura i Aplicada, Universitat Politècnica de València,
Camí de Vera s/n, 46022 València, Spain}

\email{carbivia@mat.upv.es}

%\thanks{Work supported by DGICYT Grant MTM2006-06027}

%\keywords{Integral closure of ideals, mixed multiplicities of
%ideals}
%%%%%%%%%%%%%%%%%%%%%%%%%%%%%%%%%%%%%%%%%%%%%%%%%%%%%%%%%%%%%%%%%%%%Resum
\begin{abstract}
We characterize the ideals $I$ of $\O_n$ of finite colength whose integral closure is equal to
the integral closure of an ideal generated by pure monomials. This characterization, which
is motivated by an inequality proven by Demailly and Pham \cite{DP}, is given in terms
of the log canonical threshold  of $I$ and the sequence of mixed multiplicities of $I$.
\end{abstract}

\maketitle

\section{Introduction}

%%%%%%%%Situació general del tema, introducció ràpida dels conceptes amb els quals anem a treballar i coordenades generals del tema en el que treballem a l'article.

Let $\O_n$ denote the ring of analytic function germs $f:(\C^n,0)\to \C$. Let $I$ be an ideal of $\O_n$ and
let $g_1,\dots, g_r$ be a generating system of $I$. The {\it log canonical threshold of $I$}, denoted by $\lct(I)$, is defined as the supremum of
those $s\in\R_{>0}$ such that the function $(\vert g_1\vert^2+\cdots+\vert g_r\vert^2)^{-s}$ is locally integrable around $0$.
This number, which does not depend on the chosen generating system of $I$, is always rational and
has a deep relation with other invariants (see for instance  \cite{Aluffi}, \cite{DemUppsala} or \cite{dFEM1}). 
Moreover, the log canonical threshold can be
characterized in several ways and is an object
of interest in algebraic geometry, commutative algebra an complex analytic geometry. We refer to \cite{L}, \cite{Mu2} and \cite{Totaro}
for properties and fundamental results about this number. The {\it Arnold multiplicity of $I$}, 
denoted by $\mu(I)$, is defined as $\mu(I)=\frac{1}{\lct(I)}$.

%%%%%%%%%%%%%%%%%%%%%%%Aparició de les multiplicitats mixtes

If no confusion arises, we denote by $\m$ the maximal ideal of $\O_n$. If $i\in\{1,\dots, n\}$,
then $e_i(I)$ will denote the mixed multiplicity $e(I,\dots, I,\m,\dots,\m)$, where $I$ is repeated
$i$ times and $\m$ is repeated $n-i$ times (we refer to \cite[\S17]{HS}, \cite{Rees} and \cite{Sw} for the definition and basic properties
of mixed multiplicities). We recall that $e_1(I)=\ord(I)$, where $\ord(I)=\max\{r\geq 1: I\subseteq \m^r\}$, and $e_n(I)=e(I)$, where
$e(I)$ denotes the Samuel multiplicity of $I$.

%%%%%%%%%%%%%%%%%%%%%%%Resultats i primeres ferramentes. Connexió entre lct(I) i {e(I,...,I,m,...,m):i=1,...,n}.

If $u$ is the plurisubharmonic function given by
$u=\max_j\log\vert g_j\vert$, then $e_i(I)=L_i(u)$, where $L_i(u)$ denotes
the Lelong number of the current $(dd^cu)^i$ at $0$, for $i=1,\dots, n$
(see for instance the proof of \cite[Corollary 4.2]{RashkJGA} or \cite{DemUppsala}).
Therefore, by Section 3.1 of the article \cite{DP} of Demailly and Pham, if $I$ denotes an ideal of $\O_n$ of finite colength
generated by monomials, then
\begin{equation}\label{dp}
\frac{1}{e_1(I)}+\frac{e_1(I)}{e_2(I)}+\cdots+\frac{e_{n-1}(I)}{e_n(I)}\leq \lct(I).
\end{equation}
%As a consequence of a more general result stated for plurisubharmonic functions, Rashkovskii proved in \cite[Theorem 1.4]{RashkCRASP}
%a generalization of inequality (\ref{dp}) to ideals of $\O_n$ of arbitrary codimension.
Let us denote by $\DP(I)$ the sum that appears in the left hand side of (\ref{dp}).
If $I$ is an arbitrary ideal of $\O_n$ of finite colength, then we define $\DP(I)$ in the same way.

%%%%%%%%%%%%%%%%%%El que fem nosaltres, Part 1.

In Section \ref{orderings} we show two results relating the mixed multiplicities of $I$
with the initial ideals of the powers of $I$ with respect to a specific local monomial ordering
(the negative lexicographical order).
As a direct application of these results and relation (\ref{dp}) we obtain
that $\DP(I)\leq \lct(I)$, for any ideal $I\subseteq \O_n$ of finite colength (see Remark \ref{resolem}).
This article is motivated by the question of characterizing when equality $\DP(I)=\lct(I)$ holds.

%%%%%%%%%%%%%%%%%%%%%%%Deducció del resultat previ de dFEM

We recall that, for any ideal $I\subseteq \O_n$, the following chain of inequalities holds
\begin{equation}\label{vivaHS}
\frac{1}{e_1(I)}\geq\frac{e_1(I)}{e_2(I)}\geq\cdots\geq\frac{e_{n-1}(I)}{e_n(I)},
\end{equation}
as can be seen, for instance, in \cite[Theorem 17.7.2]{HS}, \cite{RS} or \cite[p.\,41]{Tessierappendix}.
As a consequence of the inequality relating the arithmetical and the geometrical means of $n$ positive real numbers, we
immediately obtain that, if $I$ is an ideal of $\O_n$ of finite colength, then
\begin{equation}\label{meollo}
\frac{n}{e(I)^{1/n}}=n\left(\frac{1}{e_1(I)}\frac{e_1(I)}{e_2(I)}\cdots\frac{e_{n-1}(I)}{e_n(I)}\right)^{1/n}\leq
\frac{1}{e_1(I)}+\frac{e_1(I)}{e_2(I)}+\cdots+\frac{e_{n-1}(I)}{e_n(I)}\leq \lct(I).
\end{equation}
Then we have that $n\mu(I)^{1/n}\leq\lct(I)$ and equality holds if and only
if $\frac{e_1(I)}{e_2(I)}=\cdots=\frac{e_{n-1}(I)}{e_n(I)}$. It is immediate to see that this last
condition is equivalent to saying that $e_i(I)=e_1(I)^i$, for all $i=1,\dots, n$, which in turn is equivalent to the condition
$e(I)=e_1(I)^{n}=\ord(I)^n$, by (\ref{vivaHS}). We have that $I\subseteq \m^{\ord(I)}$, then the condition $e(I)=\ord(I)^n$
is equivalent to saying that $\overline I=\m^{\ord(I)}$, by the Rees' Multiplicity Theorem
(see for instance \cite[p.\,147]{HIO} or \cite[p.\,222]{HS}). Therefore it follows that $n\mu(I)^{1/n}=\lct(I)$
if and only if $\overline I=\m^{\ord(I)}$. This last equivalence was proven previously in \cite[Theorem 1.4]{dFEM1} by using another procedure.

%%%%%%%%%%%%%%%%%%El que fem nosaltres, Part 2.

Inspired by this result, we approach the problem of characterizing the equality $\DP(I)=\lct(I)$ by means of an
expression for the integral closure of $I$. For this purpose, we introduce a class of ideals that we call
{\it diagonal ideals} (see Definition \ref{elsdiagonals}).
We characterize this class in Theorem \ref{resprinc}. This theorem is supported by Corollary \ref{propxula}, where we
show a result analogous to Rees' Multiplicity Theorem using $\DP(I)$ instead of $e(I)$. As we will see (Example \ref{contraex}),
diagonal ideals are strictly contained
in the class of ideals $I\subseteq\O_n$ of finite colegth for which the equality $\DP(I)=\lct(I)$ holds.

%Moreover
%we prove that the equality $\lct(I)=\DP(I)$ forces that $e_i(I)/e_{i-1}(I)$ is an integer number, for all $i=1,\dots, n$
%(see Proposition \ref{quocdeesubi}).

%We recall (see for instance \cite[Section 17.7]{HS}) that
%$$
%\frac{1}{e_1(I)}\geq \frac{e_1(I)}{e_2(I)} \geq \cdots \geq \frac{e_{n-2}(I)}{e_{n-1}(I)}\geq \frac{e_{n-1}(I)}{e_n(I)}.
%$$
%Dir que l'objectiu de l'article és la caracterització dels ideals diagonals, en lloc de la caracterització dels ideals tals que lct(I) que alcança
%la fita de DP.

%%%%%%%%%%%%%%%%%%%%%%%%%%%%%%%%%%%%%%%%%%%%%%%%%%%%%%%%%%%%%%%%%%%%%%%%%%%%%%%%%%%%%%%%%%%%%%%%%%%%%%%%%%%%
\section{Local monomial orderings and mixed multiplicities}\label{orderings}

%Let $I$ be an ideal and let $h$ be a linear form of $\C[x_1,\dots, x_n]$ given by
%$h=a_1x_1+\cdots +a_nx_n$, where $a_1,\dots, a_n\in\C$ and $a_n\neq 0$.
%Let us consider the map $\phi:\C^{n-1}\to \C^n$ given by
%$$
%\varphi(x_1,\dots, x_{n-1})=
%(x_1,\dots, x_{n-1}, -\frac{a_1}{a_n}x_1-\cdots -\frac{a_{n-1}}{a_n}x_{n-1})
%$$
%for all $(x_1,\dots, x_{n-1})\in \C^{n-1}$. We observe that the ring morphism $\varphi^*:\O_n \to \O_{n-1}$ given by
%$\varphi^*(f)=f\circ \varphi$, for all $f\in\O_n$, induces an isomorphism between $\O_n/\langle h\rangle$
%and $\O_{n-1}$. If $I$ is an ideal of $\O_n$, then we denote by $I_h$ the image of $I$ through $\varphi^*$.

Let us fix a coordinate system $x_1,\dots, x_n\in \C^n$. If $\alpha=(\alpha_1,\dots,\alpha_n)\in\Z^n_{\geq 0}$, then
we denote the monomial $x_1^{\alpha_1}\cdots x_n^{\alpha_n}$ by $x^\alpha$. Let $\Mon_n=\{x^\alpha: \alpha\in\Z^n_{\geq 0}\}$.
Here we recall some definitions taken from \cite[Section 1.2]{GP} (see also \cite[Chapter 4, \S 3]{CLO}).
A {\it monomial ordering} in $\Mon_n$ is a total ordering $>$ on the set $\Mon_n$ such that $x^\alpha>x^\beta$ implies
$x^\gamma x^\alpha>x^\gamma x^\beta$, for all $\alpha, \beta,\gamma\in\Z^n_{\geq 0}$.

Let $>$ be a monomial ordering in $\Mon_n$. We say that $>$ is {\it local} when $1>x^\alpha$, for all $\alpha\in \Z^n_{\geq 0}$.
%If $x^\alpha>1$, for all $\alpha\in \Z^n_{\geq 0}$, then we say that $>$ is {\it global}.
In the sequel we will consider the local monomial ordering $>$ given by $x^{\alpha}> x^\beta$ if and only if
there exists some $i\in \{1,\dots, n\}$ such that $(\alpha_1,\dots, \alpha_{i-1})
=(\beta_1,\dots, \beta_{i-1})$ and $\alpha_i<\beta _i$, where $\alpha,\beta\in\Z^n_{\geq 0}$. In particular $x_n> x_{n-1}>\cdots >x_1$.
This monomial ordering is known as the {\it negative
lexicographical order} (see \cite[p.\,14]{GP}). 

If $f\in\O_n$, $f\neq 0$, let $f=\sum_k a_kx^k$ be the Taylor expansion of $f$ around the origin.
Then we define the {\it support} of $f$, denoted by $\supp(f)$, as the set of those $k\in\Z^n_{\geq 0}$ such that
$a_k\neq 0$. Therefore we denote by $\ini(f)$ the maximum of the monomials $x^k$, $k\in\supp(f)$, with respect to
the order $>$. Let us remark that, by the definition
of the negative lexicographical order, $\ini(f)$ exists. We will refer to $\ini(f)$ as the {\it initial monomial of $f$}
(in \cite{CLO} this monomial is called the {\it leading monomial of $f$} and is denoted by $\LM(f)$).

If $I$ is an ideal of $\O_n$, then we define the {\it initial ideal of $I$}, which we will denote by $\ini(I)$, as the ideal of $\O_n$
generated by all monomials $\ini(f)$ such that $f\in I$.
If $I$ has finite colength, then $\ini(I)$ has also finite colength and
$\ini(I)$ satisfies the following fundamental relation:
\begin{equation}\label{essencialini}
\dim_\C\frac{\O_n}{I}=\dim_\C\frac{\O_n}{\ini(I)}.
\end{equation}
The above result follows from \cite[Theorem 4.3, p.\,177]{CLO} (see also \cite[Corollary 7.5.6]{GP}). However, the ideals
$I$ and $\ini(I)$ do not have the same multiplicity in general, as we see in the following easy example.

\begin{ex}\label{surprise}
Let us consider the ideal $I=\langle x+y^2, y^3\rangle
\subseteq\O_2$. Using Singular \cite{Singular} we have that $\ini(I)=\langle y^2, xy, x^2\rangle$ and therefore $e(I)=3$ and $e(\ini(I))=4$.
We also observe that $e_1(I)=1$ and $e_1(\ini(I))=2$.
\end{ex}

\begin{prop}\label{semicont}
Let $I$ be an ideal of $\O_n$ of finite colength. Then $e_j(I)\leq e_j(\ini(I))$, for all $j=1,\dots, n$, and $\lct(\ini(I))\leq \lct(I)$.
\end{prop}

\begin{proof}
Let us consider the coordinates $(x_1,\dots, x_n,t)$ in $\C^{n+1}$. Since we suppose that $I$ has finite colength, then $I$ admits
a generating system formed by polynomials. In particular, by \cite[Corollary 7.4.6]{GP} and \cite[Corollary 7.5.2]{GP},
there exists an ideal $J\subseteq \O_{n+1}$ generated by homogeneous polynomials verifying the following properties:
\begin{enumerate}
\item $J_0=\ini(I)$ and $J_1=I$, where we denote
by $J_t$ the ideal of $\O_n$ obtained by fixing the variable $t$ in each element of $J$;
\item $\O_{n+1}/J$ is a flat $\C[t]$-algebra;
\item the rings $\O_n/J_t$ and $\O_n/I$ are isomorphic, for all $t\in\C\smallsetminus\{0\}$.
\end{enumerate}
By the lower semicontinuity of the log canonical threshold (see \cite{DK} or \cite[Corollary 9.5.39]{L}), we have
$\lct(J_0)\leq \lct(J_t)=\lct(I)$, for all $t$ small enough, $t\neq 0$, where the equality $\lct(J_t)=\lct(I)$
follows by the existence of a ring isomorphism $\O_n/J_t\simeq \O_n/I$, for all $t\in \C\smallsetminus\{0\}$.

%By the upper semicontinuity of mixed multiplicities we obtain $\e_j(J_t)$

Let us fix an integer $j\in\{1,\dots, n\}$.
We recall that $e_j(J_0)=e(J_0,\dots, J_0,\m,\dots,\m)$, where $J_0$ is repeated $j$ times and $\m$ is repeated $n-j$ times.
Hence, by \cite[Theorem 17.4.9]{HS} (see also \cite[Corollaire 2.2]{Cargese}), 
the mixed multiplicity $e_j(I)$ is expressed as
\begin{equation}\label{glforms}
e_j(J_0)=e\left(J_0\frac{\O_n}{\langle h_{j+1}, \dots, h_n\rangle}\right),
\end{equation}
for generic linear forms $h_{j+1}, \dots, h_n$ in $\C[x_1,\dots, x_n]$ (this set of linear forms is empty when $j=n$). 

Then, let us fix linear forms $h_{j+1}, \dots, h_n\in\C[x_1,\dots, x_n]$ such that relation (\ref{glforms}) holds.
By the upper semicontinuity of Samuel multiplicity (see \cite[p.\,547]{GK} or \cite[p.\,126]{Lipman}) we have
\begin{equation}\label{ejs}
e_j(J_0)\geq e\left(J_t\frac{\O_n}{\langle h_{j+1}, \dots, h_n\rangle}\right)\geq e_j(J_t)
\end{equation}
where the second inequality follows from \cite[Theorem 17.4.9]{HS}.

The existence of a ring isomorphism $\O_n/J_t\simeq \O_n/I$, for all $t\in \C\smallsetminus\{0\}$, implies
that there exists a biholomorphism $\varphi_t:(\C^n,0)\to (\C^n,0)$
such that $\varphi_t^*(I)=J_t$ (see \cite[p.\,16]{Fischer} o \cite[p.\,57]{GLS}). In particular, we obtain that $e_j(I)=e_j(J_t)$, for all $t\neq 0$.
Then, since $J_0=\ini(I)$, we have that $e_j(\ini(I))\geq e_j(I)$, for all $j=1,\dots, n$, by virtue of (\ref{ejs}).
\end{proof}

Let $\LL\subseteq\{1,\dots, n\}$, $\LL\neq\emptyset$. We define $\C^n_\LL=\{x\in\C^n: x_i=0,\,\textrm{for all $i\notin\LL$}\}$ and
we denote by $\pi_\LL$ the natural projection $\C^n\to \C^n_\LL$. Let us denote by $\O_{n, \LL}$ the subring of $\O_n$
formed by all functions germs of $\O_n$ depending at most on the variables $x_i$ with $i\in \LL$.
Let $f\in\O_n$ and let us suppose that the Taylor expansion of $f$ around the origin is given by $f=\sum_ka_kx^k$.
Then we denote by $f_\LL$ the sum of all terms $a_kx^k$ such that $k\in\supp(f)\cap \R^n_\LL$. 
If $J$ is an ideal of $\O_n$ then we denote by $J_{\LL}$ the ideal of $\O_{n,\LL}$ generated by all elements $f_\LL$,
where $f\in J$.

\begin{lem}\label{lesLfora}
Let $J$ be an ideal of $\O_n$ of finite colength and let $\LL=\{j,\dots,n\}$, for some $j\in\{1,\dots, n\}$. Then
$$
\ini(J_\LL)=\ini(J)_\LL.
$$
\end{lem}

\begin{proof}
If $j=1$, there is nothing to prove, so let us suppose that $j>1$.
Let $f\in J$, such that $\ini(f)_\LL\neq 0$. In particular, it follows that $f_\LL\neq 0$ and $\supp(\ini(f))\subseteq \supp(f_\LL)$,
that is, $\ini(f)_\LL=\ini(f)=\ini(f_\LL)$. Then $\ini(J_\LL)\supseteq \ini(J)_\LL$.

On the other hand, let $f\in J$ such that $f_\LL\neq 0$. Then there exists
some element $k\in\supp(f)$ such that $k_1=\cdots=k_{j-1}=0$.
Hence $x^{k}>x^{k'}$, for all $k'\in\supp(f)$ such that
$k'_i\neq 0$, for some $i\notin \LL$, by the definition of the negative lexicographical order. In particular
$\supp(\ini(f))\subseteq \supp(f)\cap \R^n_\LL$ and hence
$\ini(f)=\ini(f_\LL)$. In particular $\ini(f_\LL)=\ini(f)_\LL$. Therefore $\ini(J_\LL)\subseteq\ini(J)_\LL$.
\end{proof}

If $\varphi:\C^n\to \C^n$ is a linear change of coordinates and $J$ is an ideal of $\O_n$,
then we denote by $\varphi^*(J)$ the ideal of $\O_n$ generated by the elements $g\circ\varphi$, where $g\in J$.

\begin{thm}\label{mustatagin}
Let $I$ be an ideal of $\O_n$ of finite colength. Then, for all $j\in\{1,\dots, n\}$, we have
\begin{equation}\label{vivaej}
e_j(I)=\lim_{t\to +\infty}\frac{e_j\left(\ini(\varphi^*(I)^t)\right)}{t^j}
\end{equation}
for a generic linear change of coordinates $\varphi:\C^n\to \C^n$.
\end{thm}

\begin{proof}
%The case $j=n$ is a result of Musta\c{t}\u a \cite[Corollary 1.13]{Mu1}. Let us suppose that $j>1$.
By \cite[Theorem 17.4.9]{HS}, there exist generic linear forms $h_1,\dots, h_n\in\C[x_1,\dots, x_n]$ such that
$$
e_j(I)=e\left(I\frac{\O_n}{\langle h_{1},\dots, h_{n-j}\rangle}\right)
$$
for all $j=1,\dots, n$ (where we consider that this set of linear forms is empty when $j=n$).
Let $\varphi:\C^n\to \C^n$ be the linear change of coordinates such that $h_i\circ \varphi=x_i$, for all
$i=1,\dots, n$. Let us denote by $J$ the ideal $\varphi^*(I)$. Then
\begin{equation}\label{prepare}
e_j(I)=e\left(I\frac{\O_n}{\langle h_{1},\dots, h_{n-j}\rangle}\right)=e\left(J\frac{\O_n}{\langle x_{1},\dots, x_{n-j}\rangle}\right)=e(J_{\LL})
\end{equation}
where $\LL=\{n-j+1,\dots, n\}$ and $e(J_{\LL})$ denotes the Samuel multiplicity of $J_\LL$ in the ring $\O_{n,\LL}$.
By \cite[Corollary 1.13]{Mu1} (see also \cite[Theorem 1.1]{Cut}) we have that
$$
e(J_{\LL})=\lim_{t\to +\infty}\frac{e\left(\ini(J^t_\LL)\right)}{t^j}
$$
where $\ini(J^t_\LL)$ is the initial ideal of $J^t_\LL$ with respect to the negative lexicographical ordering in the
monomials of $\O_{n,\LL}$, for all $t\in\Z_{\geq 1}$.
By Lemma \ref{lesLfora} we have $e(\ini(J^t_\LL))=e\left(\ini(J^t)_\LL\right)$. Moreover
$e(\ini(J^t)_\LL)\geq e_{j}(\ini(J^t))\geq e_{j}(J^t)$, where the first inequality follows from  \cite[Theorem 17.4.9]{HS}
and the second inequality is an application of Proposition \ref{semicont}. Putting this information together we obtain
the following chain of inequalities:
$$
e\left(\ini(J^t_\LL)\right)=e\left(\ini(J^t)_\LL\right)\geq e_{j}\left(\ini(J^t)\right)\geq e_{j}(J^t)=t^je_j(J).
$$
Then, dividing each term of the previous inequalities by $t^j$ and taking limits,
we arrive to
\begin{align*}
e_j(I)=e(J_{\LL})&=\lim_{t\to +\infty}\frac{e\left(\ini(J^t_\LL)\right)}{t^j}=\lim_{t\to +\infty}\frac{e\left(\ini(J^t)_\LL\right)}{t^j}\\
& \geq \lim_{t\to +\infty}\frac{e_j\left(\ini(J^t)\right)}{t^j}= \lim_{t\to +\infty}\frac{e_j\left(\ini(\varphi^*(I)^t)\right)}{t^j}\\
&\geq \lim_{t\to +\infty}\frac{e_j(\varphi^*(I)^t)}{t^j}=e_j(\varphi^*(I))=e_j(I).
\end{align*}
Then the result follows.
\end{proof}

\begin{rem}
By the argument of the proof of the previous result, if we fix an index $j\in\{1,\dots, n-1\}$, and
$h_1,\dots, h_{n-j}$ are linear forms of $\C[x_1,\dots, x_n]$ such that $e_j(I)$ coincides with the multiplicity of $I$
in the quotient ring $\O_n/\langle h_1,\dots, h_{n-j}\rangle$, then relation (\ref{vivaej}) holds by taking
$\varphi:\C^n\to \C^n$ as any linear change of coordinates such that $h_i\circ \varphi=x_i$, for all $i=1,\dots, j$.
\end{rem}

\begin{cor}\label{limDPs}
Let $I$ be an ideal of finite colength of $\O_n$. Then
\begin{equation}\label{nadal2015}
\DP(I)=\lim_{t\to+\infty} t\DP\left(\ini(\varphi^*(I)^t)\right).
\end{equation}
for a generic linear change of coordinates $\varphi:\C^n\to \C^n$.
\end{cor}

\begin{proof} Let us fix a generic change of coordinates $\varphi:\C^n\to \C^n$ and let us denote
the ideal $\varphi^*(I)$ by $J$. Then,
for any $t\in\Z_{\geq 1}$, we have
\begin{align*}
t\DP\left(\ini(J^t)\right)&=
t\frac{1}{e_1(\ini(J^t))}+t\frac{e_1(\ini(J^t))}{e_2(\ini(J^t))}+\cdots+t\frac{e_{n-1}(\ini(J^t))}{e_n(\ini(J^t))}\\
&=\frac{1}{e_1(\ini(J^t))/t}+\frac{e_1(\ini(J^t))/t}{e_2(\ini(J^t))/t^2}+\cdots+\frac{e_{n-1}(\ini(J^t))/t^{n-1}}{e_n(\ini(J^t))/t^n}
\end{align*}
By Theorem \ref{mustatagin} and the definition of $\DP(I)$ we immediately obtain the desired result.
\end{proof}

\begin{rem}\label{resolem}
Let $I\subseteq \O_n$ be any ideal of finite colength.
Let $\varphi:\C^n\to \C^n$ be a linear change of coordinates such that relation (\ref{nadal2015}) holds
for $\varphi$ and $I$. As a direct consequence
of \cite[\S 3.1]{DP} we obtain that $\DP(K)\leq \lct(K)$, for any
monomial ideal $K$ of $\O_n$ of finite colength.
Then
$$
\DP\left(\ini(\varphi^*(I)^t)\right)\leq \lct\left(\ini(\varphi^*(I)^t)\right)
$$
for all $t\in\Z_{\geq 1}$. In particular
\begin{equation}\label{tlctin}
t\DP\left(\ini(\varphi^*(I)^t)\right)\leq t\,\lct\left(\ini(\varphi^*(I)^t)\right)\leq t\lct(\varphi^*(I)^t)=\lct(I)
\end{equation}
for all $t\in \Z_{\geq 1}$, were the second inequality follows from Proposition \ref{semicont}.
Therefore, taking limits when $t\to \infty$ in all parts of the previous inequalities,
we obtain
\begin{equation}\label{DPcentral}
\DP(I)\leq \lct(I)
\end{equation}
as a consequence of Corollary \ref{limDPs}. Then we have shown an alternative approach to the
deduction of (\ref{DPcentral}) as a corollary of the analogous inequality for monomial ideals.

To the best of our knowledge, the proof of (\ref{DPcentral}) as a corollary of the analogous result for monomial ideals
explained in \cite[\S 3.3]{DP} relies on the equality $e_j(I)=e_j(\ini(I))$, for all $j=1,\dots, n$, where $\ini(I)$ denotes the initial
ideal of $I$ with respect to any monomial order. However, as shown in Example \ref{surprise}, the ideals
$I$ and $\ini(I)$ do not have the same set of mixed multiplicities in general.

Let us also point out that if $\DP(I)=\lct(I)$, then relation (\ref{tlctin}) and Corollary \ref{limDPs} show that
$$
\lct(I)=\lim_{t\to +\infty}t \lct\left(\ini(\varphi^*(I)^t)\right).
$$

\end{rem}

%%%%%%%%%%%%%%%%%%%%%%%%%%%%%%%%%%%%%%%%%%%%%%%%%%%%%%%%%%%%%%%%%%%%%%%%%%%%%%%%%%%%%%%%%%%%%%%%%%%%%%%%%%%%
\section{Mixed multiplicities and diagonal ideals}

Let us fix along the remaining text a coordinate system $x_1,\dots,x_n$ in $\C^n$, unless otherwise stated.
Let $I$ be an ideal of $\O_n$. We denote the integral closure of $I$ by $\overline I$ and the Newton polyhedron of $I$ by $\Gamma_+(I)$.
Let us recall that $\Gamma_+(I)$ is the smallest convex set of $\R^n_+$ containing the supports of the elements of $I$.
Therefore $\Gamma_+(I)$ is equal to the convex hull of the set $\{k+v: k\in\supp(f),\,f\in I,\,v\in\R^n_{\geq 0}\}$.
In general it holds that $\Gamma_+(I)=\Gamma_+(\overline I)$ (see \cite[p.\,399]{CBAJoint}).
If $I$ admits a generating system formed by monomials, then we say that $I$ is a {\it monomial ideal}.

We define the {\it term ideal} of $I$ as the ideal generated by all the monomials $x^k$ such that $k\in\Gamma_+(I)$.
We will denote this ideal by $I^0$. If $I$ is a monomial ideal, then $\overline I$ is also monomial and therefore $\overline I=I^0$
(see \cite[p.\,11]{HS} or \cite{McZ});
however the converse is not true,
as is shown by the ideal $I$ of $\O_2$ given by $I=\langle x^2+y^2, xy\rangle$.
The ideals $I$ for which $\overline I$ is generated by monomials are called {\it Newton non-degenerate ideals} (see \cite{CBAJoint} or \cite{Tessiermbp}).

\begin{defn}\label{elsdiagonals}
Let $I$ be an ideal of $\O_n$. We say that $I$ is {\it diagonal} when there exist positive integers $a_1,\dots, a_n$ such that
$\overline I=\overline{\langle x_1^{a_1},\dots, x_n^{a_n}\rangle}$.
\end{defn}

Then any power of the maximal ideal of $\O_n$ is a diagonal ideal. Moreover, any diagonal ideal is Newton non-degenerate.
As a consequence of the previous definition, if $I$ is diagonal then $I^0$ is also, but the converse is not true, as is shown by the
ideal of $\O_2$ given by $I=\langle x+y, x^2 \rangle\subseteq\O_2$ (in this case $I^0$ is equal to the maximal ideal).

Let $I$ be an ideal of $\O_n$ of finite colength. Then by virtue of (\ref{DPcentral}) and the inclusion $I\subseteq I^0$ we have the inequalities
\begin{equation}\label{meollo2}
\DP(I)\leq \lct(I)\leq \lct(I^0).
\end{equation}
We recall the following result of Howald \cite{Howald} (see also \cite{McZ}), where $\lct(I)$ is characterized in terms of a combinatorial
characteristic of $\Gamma_+(I)$ if $I$ is
a monomial ideal.

\begin{thm}\label{resHowald}\cite{Howald}
Let $I$ be a monomial ideal of $\O_n$. Then
$$
\lct(I)=\frac{1}{\min\big\{\mu>0: \mu(1,\dots, 1)\in\Gamma_+(I)\big\}}.
$$
\end{thm}

\begin{prop}\label{propxula0}
Let $D\subseteq \R^n_{>0}$ be the set defined by
\begin{equation}\label{Dxulo}
D=\{(t_1,\dots, t_n)\in\,\R^n_{>0}: t_1^2\leq t_2,\, t_j^2\leq t_{j-1}t_{j+1},\,\textnormal{for all $j=2,\dots, n-1$}\}.
\end{equation}
Let us consider the function $f:\R^n_{>0}\to \R$ given by
$$
f(t_1,\dots, t_n)=\frac{1}{t_1}+\frac{t_1}{t_2}+\dots \frac{t_{n-1}}{t_n},
$$
for all $(t_1,\dots, t_n)\in \R^n_{>0}$. Let $a=(a_1,\dots, a_n), b=(b_1,\dots, b_n)\in D$ such that
$a_i\leq b_i$, for all $i=1,\dots, n$.  Then $f(a)\geq f(b)$ and equality holds
 only if and only if $a=b$.
\end{prop}

\begin{proof}
Let us see first that $D$ is convex.
For all $j=1,\dots, n$,
we define $D_j=\{(t_1,\dots, t_n)\in\,\R^n_{>0}: t_j^2\leq t_{j-1}t_{j+1}\}$, where we set $t_0=1$, for all $t=(t_1,\dots, t_n)\in\R^n_{>0}$.
Then it suffices to see that $D_j$ is convex, for all $j=1,\dots, n$, since $D=D_1\cap \cdots \cap D_n$.

Let us fix an index $j\in\{1,\dots, n\}$. Let $s=(s_1,\dots, s_n)$ and $t=(t_1,\dots, t_n)$ be elements of $D_j$ and let $\lambda\in[0,1]$.
We define $\bfu=(\sqrt{\lambda s_{j-1}}, \sqrt{(1-\lambda) t_{j-1}})$ and $\mathbf v=(\sqrt{\lambda s_{j+1}}, \sqrt{(1-\lambda) t_{j+1}})$.
Let us denote by $\bfu\cdot\bfv$ the usual scalar product of $\bfu$ and $\bfv$. By
applying the definition of $D_j$ and the Cauchy-Schwarz inequality we find that
\begin{align}
(\lambda s_j+(1-\lambda)t_j )^2&\leq \left(\lambda \sqrt{s_{j-1}}\sqrt{s_{j+1}}+(1-\lambda)\sqrt{t_{j-1}}\sqrt{t_{j+1}}\right )^2    \label{CS1}\\
&= \left( \bfu\cdot\bfv\right)^2\leq \Vert\bfu\Vert^2\Vert \bfv\Vert^2=(\lambda s_{j-1}+(1-\lambda)t_{j-1} )(\lambda s_{j+1}+(1-\lambda)t_{j+1} ).\label{CS2}
\end{align}
Then $\lambda s+(1-\lambda)t\in D_j$, for all $\lambda\in [0,1]$, and hence $D_j$ is convex. Therefore $D$ is convex.

 The proof of the inequality $f(a)\geq f(b)$ is contained in the
proof of \cite[Lemma 3.1]{DP}, however we reproduce it for the sake of completeness and for its implications
in the proof of the second part of the result.

Let us consider the
function $g:[0,1]\to \R_{>0}$ defined by $g(\lambda)=f(a+\lambda (b-a))$, for all $\lambda\in [0,1]$.
We observe that
\begin{equation}\label{pardef}
\frac{\partial f}{\partial t_1}(t)=-\frac{1}{t_1^2},\hspace{1cm}
\frac{\partial f}{\partial t_j}(t)=-\frac{t_{j-1}}{t_j^2}+\frac{1}{t_{j+1}}, \hspace{1cm}
\frac{\partial f}{\partial t_n}(t)=-\frac{t_{n-1}}{t_n^2}
\end{equation}
for all $t\in \R^n_{>0}$ and all $j=2,\dots, n-1$. In particular, we have $\frac{\partial f}{\partial t_j}(t)\leq 0$,
for all $t\in D$ and all $j=1,\dots, n$. Then
\begin{equation}\label{lag}
g'(\lambda)=\sum_{j=1}^n\left(\frac{\partial f}{\partial t_j}(a+\lambda(b-a))\right)(b_j-a_j)\leq 0
\end{equation}
for all $\lambda\in \,]0,1[$. Hence $g$ is a decreasing function, which implies that $f(a)\geq f(b)$.

Let us suppose that $f(a)=f(b)$, which means that $g(0)=g(1)$.
Then there exists some $\l0\in\,]0,1[$ such that $g'(\l0)=0$, by the Mean Value Theorem.
Let $c_0=a+\lambda_0(b-a)\in D$.
By
(\ref{lag}) and the fact that $\frac{\partial f}{\partial t_j}(t)\leq 0$,
for all $t\in D$ and all $j=1,\dots, n$, we conclude that
\begin{equation}\label{pardef2}
\frac{\partial f}{\partial t_1}(c_0)(b_1-a_1)=0,\hspace{1cm}
\frac{\partial f}{\partial t_j}(c_0)(b_j-a_j)=0,\hspace{1cm}
\frac{\partial f}{\partial t_n}(c_0)(b_n-a_n)=0
\end{equation}
for all $j=2,\dots, n-1$.

Let us suppose that $a_n\neq b_n$. Then (\ref{pardef2}) implies that
$$
\frac{\partial f}{\partial t_n}(a+\l0(b-a))(a_n-b_n)=
-\frac{a_{n-1}+\l0(b_{n-1}-a_{n-1})}{(a_{n}+\l0(b_{n}-a_{n}))^2}(b_n-a_n)=0.
$$
Then $\lambda=-b_{n-1}/(a_{n-1}-b_{n-1})$, which contradicts the hypothesis that $\lambda\in\,]0,1[$. Therefore
$a_n=b_n$.

If we assume that $a_{n-1}\neq b_{n-1}$, by (\ref{pardef}) and (\ref{pardef2}), we conclude that
\begin{equation}\label{eql0}
(\lambda_0 b_{n-1} +(1-\l0)a_{n-1})^2=(\lambda_0 b_{n-2} +(1-\l0)a_{n-2})(\lambda_0 b_{n} +(1-\l0)a_{n}).
\end{equation}
We observe that, by inequality (\ref{CS1}), this condition can not hold if $a_j^2<a_{j-1}a_{j+1}$ or $b_j^2<b_{j-1}b_{j+1}$. So
(\ref{eql0}) forces that
\begin{equation}\label{eqfr}
a_{n-1}^2=a_{n-2}a_{n} \hspace{0.5cm}\textnormal{and}\hspace{0.5cm} b_{n-1}^2=b_{n-2}b_{n}.
\end{equation}

By (\ref{CS1}), (\ref{CS2}) and the characterization of equality in the Cauchy-Schwarz inequality, condition (\ref{eql0}) is equivalent to saying that
$$
\frac{\sqrt{(1-\l0)a_{n}}}{\sqrt{(1-\l0)a_{n-2}}}=\frac{\sqrt{(1-\l0)b_{n}}}{\sqrt{(1-\l0)b_{n-2}}},
$$
which in turn is equivalent to saying that
%$
%\frac{a_{n}}{a_{n-2}}=\frac{b_{n}}{b_{n-2}}.
%$
$
{a_{n}}/{a_{n-2}}={b_{n}}/{b_{n-2}}.
$
Then, since $a_n=b_n$, we obtain that $a_{n-2}=b_{n-2}$. Hence $a_{n-1}=b_{n-1}$, by (\ref{eqfr}), and thus we arrive to a contradiction.
The remaining equalities $a_j=b_j$, for all $j=1,\dots, n-2$, follow analogously. Then the result is proven.
\end{proof}

%Let $I$ be a proper ideal of finite colength of $\O_n$.
%Then $\Gamma_+(I)$ intersects each coordinate axis of $\R^n$ in a point different from the origin.
%In this case we define $p_i(I)=\min\{r\geq 0: r\e_i\in\Gamma_+(I)\}$, for $i=1,\dots, n$, where $\e_1,\dots, \e_n$ denotes the canonical basis of $\R^n$.
%We also define $p_0(I)=\min\{p_1(I),\dots, p_n(I)\}$.
%Let us remark that $\ord(I)\leq p_0(I)$ in general and equality holds if $I$ is diagonal.

\begin{cor}\label{propxula}
Let $R$ be a Noetherian local ring. Let us suppose that $R$ is quasi-unmixed.
Let $I_1,I_2$ be two ideals of finite colength of $R$ such that $I_1\subseteq I_2$. Then
\begin{equation}\label{esdeDP}
\DP(I_1)\leq\DP(I_2)
\end{equation}
and equality holds if and only if $\overline{I_1}=\overline{I_2}$.
\end{cor}

\begin{proof}
If $\overline{I_1}=\overline{I_2}$, then $e_i(I_1)=e_i(I_2)$, for all $i=1,\dots, n$ (see \cite[\S 17.4]{HS} or \cite[p.\,306]{Cargese}) and hence $\DP(I_1)=\DP(I_2)$.
Let $D$ be the set defined in Proposition \ref{propxula0}.
Let us consider the vectors
\begin{equation}\label{aib}
a=(e_1(I_2),\dots, e_n(I_2)),\hspace{0.5cm}   b=(e_1(I_1),\dots, e_n(I_1)).
\end{equation}
Since $I_1\subseteq I_2$, then $e_i(I_2)\leq e_i(I_1)$, for all $i=1,\dots, n$.
Moreover, the vectors $a$ and $b$ defined in (\ref{aib}) belong to $D$, by (\ref{vivaHS}).
Then we can apply Proposition \ref{propxula0} to deduce that $\DP(I_1)\leq \DP(I_2)$ and
equality holds if and only if $e_i(I_1)=e_i(I_2)$, for all $i=1,\dots, n$. In particular
$\DP(I_1)= \DP(I_2)$ implies $e(I_1)=e(I_2)$. The equality $e(I_1)=e(I_2)$ together with
the inclusion $I_1\subseteq I_2$ implies that
$\overline{I_1}=\overline{I_2}$ by the Rees' Multiplicity Theorem \cite[p.\,222]{HS}.
\end{proof}

If $f\in \O_n$, then we denote by $J(f)$ the ideal of $\O_n$ generated by $\frac{\partial f}{\partial x_1},
\dots, \frac{\partial f}{\partial x_n}$. Let us suppose that $f$ has an isolated singularity at the origin, that is, the ideal $J(f)$
has finite colength in $\O_n$. Let $\mu^*(f)$ denote the vector $(\mu^{(1)}(f),\dots,\mu^{(n)}(f))$, where
$\mu^{(i)}(f_t)$ denotes the Milnor number of the restriction of $f_t$ to a generic plane of dimension $i$ in $\C^n$ passing through the origin,
for all $i=1,\dots, n$ (see \cite[\S\,1]{Cargese}).

We say that a given property
$(P_t)$ holds for all $\vert t\vert \ll 1$ if there exists
an open ball $U$ centered at $0$ in $\C$ such that the property $(P_t)$ holds whenever $t\in U$.

\begin{cor}
Let $f_t:(\C^n,0)\to (\C,0)$ be an analytic deformation such that $f_t$ has an isolated
singularity at the origin, for all $\vert t\vert \ll 1$.
Then
\begin{enumerate}
\item $\DP(J(f_t))$ is lower semicontinuous, that is, $\DP(J(f_0))\leq\DP(J(f_t))$, for all $\vert t\vert \ll 1$
\item $\DP(J(f_t))$ is constant, for $\vert t\vert \ll 1$, if and only if $\mu^*(f_t)$ is constant, for $\vert t\vert \ll 1$.
\end{enumerate}
\end{cor}

\begin{proof}
By the results of Teissier in \cite[\S 1]{Cargese}, it is well known that $\mu^{(i)}(f_t)=e_i(J(f_t))$, where $\mu^{(i)}(f_t)$ denotes the
Milnor number of the restriction of $f_t$ to a generic plane of dimension $i$ in $\C^n$ passing through the origin. Since Milnor numbers
are upper semicontinuous (see \cite[Theorem 2.6]{GLS}), we conclude that $e_i(J(f_t))\leq e_i(J(f_0))$, for all $i=1,\dots, n$. Then both items of the result follow as an immediate
consequence of Proposition \ref{propxula0}.
\end{proof}

%\begin{cor}
%Let $F:(\C\times\C^n,0)\to (\C^m,0)$ be an analytic map. Let us write $F(t,x)=(F_{1,t}(x),\dots, F_{m,t}(x))$, for all $(x,t)$
%belonging to some open neighbourhood of the origin. Let $I_t$ denote the ideal of $\O_n$ generated by the germs
%$F_{1,t},\dots, F_{m,t}$, for all $\vert t\vert\ll 1$. Let us suppose that $I_t$ has finite colength, for all $\vert t\vert \ll 1$.
%Then
%\begin{enumerate}
%\item $\DP(I_t)$ is lower semicontinuous, that is, $\DP(I_0)\leq\DP(I_t)$, for all $\vert t\vert \ll 1$
%\item $\DP(I_t)$ is constant, for $\vert t\vert \ll 1$, if and only if $e_i(I_t)$ is constant, for all $i=1,\dots, n$, for $\vert t\vert \ll 1$.
%\end{enumerate}
%\end{cor}
%
%\begin{proof}
%It is known that

\begin{thm}\label{resprinc}
Let $I$ be an ideal of $\O_n$ of finite colength.  Then the following conditions are equivalent:
\begin{enumerate}
\item[(a)] $I$ is diagonal.
\item[(b)] $\lct(I^0)=\DP(I)$.
\end{enumerate}
\end{thm}

\begin{proof} The implication $(\textrm{a})\Rightarrow (\textrm{b})$ is a direct consequence of Theorem \ref{resHowald} and the
known equality $\lct(I)=\lct(\overline I)$ (see \cite[\S 11.1]{L}).

Let us prove $(\textrm{b})\Rightarrow (\textrm{a})$.
Let us suppose first that $I$ is an ideal generated by monomials such that $\lct(I^0)=\DP(I)$.
Hence $\lct(I)=\frac 1\mu_0$, where
$\mu_0=\min\{\mu>0: \mu\e\in\Gamma_+(I)\}$ and $\e=(1,\dots, 1)$, by Theorem \ref{resHowald}.
Let $\pi$ denote a supporting hyperplane
of $\Gamma_+(I)$ containing the point $\mu_0\e$ and defined by the zeros of a linear form with rational coefficients.
Let us write the equation of $\pi$ as
$$
\frac{x_1}{c_1}+\dots+\frac{x_n}{c_n}=1
$$
where $c_1,\dots, c_n\in\Q_{>0}$.
If necessary, we can reorder the variables to obtain $c_1\leq \dots\leq c_n$.
Let $r$ be a positive integer such that $rc_1,\dots, rc_n\in\Z_{\geq 1}$ and let us
denote by $H$ the ideal of $\O_n$ generated by $x_1^{rc_1},\dots, x_n^{rc_n}$. Since $\pi$ is a supporting hyperplane of $\Gamma_+(I)$
passing through the point $\mu_0\bf e$, we have $I^r\subseteq \overline H$ and $\lct(I^r)=\lct(H)$.
Moreover $e_i(H)=r^ic_1\cdots c_i$, for all $i=1,\dots, n$, since $c_1\leq \cdots \leq c_n$.
Therefore
\begin{equation}\label{elsH1}
\lct(I^r)=\lct(H)=\frac{1}{rc_1}+\frac{1}{rc_2}+\cdots+\frac{1}{rc_n}=\frac{1}{e_1(H)}+\frac{e_1(H)}{e_2(H)}+\cdots +\frac{e_{n-1}(H)}{e_n(H)}=\DP(H).
\end{equation}
Since $\overline I=I^0$, we have that $\lct(I)=\DP(I)$, by hypothesis. Thus
\begin{equation}\label{elsH2}
\lct(I^r)=\frac 1r\lct(I)=\frac 1r \DP(I)=\DP(I^r),
\end{equation}
where the last equality follows from the relation $e_i(I^r)=r^ie_i(I)$, for all $i=1,\dots, n$ (see \cite[Proposition 17.5.1]{HS}).
Then (\ref{elsH1}) and (\ref{elsH2}) show that $\DP(I^r)=\DP(H)$ and, by Corollary \ref{propxula}, we obtain that $\overline{I^r}=\overline H$.
Thus $r\Gamma_+(I)=\Gamma_+(I^r)=\Gamma_+(H)$, which
implies that $\Gamma_+(I)$ has a unique compact face $\Delta$ of dimension $n-1$. Since the vertexes of $\Gamma_+(I)$ are contained in
$\Z^n_{\geq 1}$, we conclude that we can take $r=1$ and that, in this case, the hyperplane $\pi$ contains $\Delta$. Consequently
$c_i\in\Z_{\geq 1}$ and $c_i=e_i(I)/e_{i-1}(I)$, for all
$i=1,\dots, n$. Hence $\overline I=\overline{\langle x_1^{c_1},\dots, x_n^{c_n}\rangle}$, which means that $I$ is diagonal.

Let $I$ be an arbitrary ideal of $\O_n$ of finite colength such that $\lct(I^0)=\DP(I)$.
Then, by a direct application of (\ref{meollo2}) and Corollary \ref{propxula} we obtain
the following chain of inequalities
\begin{align}
\DP(I)&=\lct(I)=\lct(I^0)\geq \frac{1}{e_1(I^0)}+\frac{e_1(I^0)}{e_2(I^0)}+\cdots +\frac{e_{n-1}(I^0)}{e_n(I^0)}\label{primera}\\
&\geq \frac{1}{e_1(I)}+\frac{e_1(I)}{e_2(I)}+\cdots +\frac{e_{n-1}(I)}{e_n(I)}=\DP(I).\label{segona}
\end{align}

Hence we deduce that $\lct(I^0)=\DP(I^0)$, which implies, by the case analyzed before, that $I^0$ is a diagonal ideal.
Moreover (\ref{primera}) and (\ref{segona}) also show that $\DP(I)=\DP(I^0)$. Then $\overline I=\overline{I^0}$, by Corollary \ref{propxula},
and consequently $I$ is a diagonal ideal.
\end{proof}

%\begin{prop}\label{fitadelesPi}
%Let $I$ be an ideal of $\O_n$ of finite colength.
%Then
%$$
%\frac{e_i(I)}{e_{i-1}(I)}\leq \LL_0^{(i)}(I)
%$$
%for all $i=1,\dots, n$. If, in addition, we suppose that $I$ is a monomial ideal and $p_1(I)\leq \cdots \leq p_n(I)$, then $\LL_0^{(i)}(I)\leq p_i(I)$, for all %$i=1,\dots, n$.
%\end{prop}

%\begin{proof}
%Let $I^0$ denote the ideal of $\O_n$ generated by all monomials $x^k$ such that $k\in\Gamma_+(I)$.
%Then, for all $i=1,\dots, n$, we have
%$$
%\frac{e_i(I)}{e_{i-1}(I)}\leq \frac{e_i(I^0)}{e_{i-1}(I^0)}\leq \LL_0^{(i)}(I^0)\leq p_i(I).
%$$
%\end{proof}

%\begin{cor}
%Let $I$ be a monomial ideal of $\O_n$ of finite colength. Let us suppose that $p_1(I)\leq \cdots \leq p_n(I)$. Then
%$$
%\lct(I)\geq \frac{1}{\ord(I)}+\frac{1}{p_2(I)}+\cdots +\frac{1}{p_n(I)}
%$$
%and equality holds if and only if $\ord(I)=p_1(I)$ and
%$$
%\overline I=\overline{\langle x_1^{p_1(I)},\dots, x_n^{p(n)}\rangle}.
%$$
%\end{cor}

%\begin{proof}
%It follows as an immediate application of Proposition \ref{fitadelesPi} and Theorem \ref{resprinc}
%\end{proof}

\begin{rem}
{(i)} We observe that condition (b) of Theorem \ref{resprinc} is equivalent to impose the conditions
$\lct(I)=\DP(I)$ and $\lct(I)=\lct(I^0)$, by (\ref{meollo2}). In general the condition $\lct(I)=\DP(I)$ does not imply $\lct(I)=\lct(I^0)$ and
hence it does not force the ideal $I$ to be diagonal, as is shown in Example \ref{contraex}. Obviously, the condition $\lct(I)=\lct(I^0)$
holds if $\overline I$ is a monomial ideal. If $I$ is an arbitrary ideal of $\O_n$, let us denote by $K_I$ the ideal of $\O_n$
generated by all the monomials $x^k$ such that $x^k\in\overline I$. Then $\lct(K_I)\leq \lct(I)\leq \lct(I^0)$. If we suppose that
$\mu(I^0)\in\Z_{\geq 1}$ and the monomial
$(x_1\cdots x_n)^{\mu(I^0)}$ is integral over $I$, then we have $\lct(K_I)=\lct(I^0)$, by Theorem \ref{resHowald}, and then $\lct(I)=\lct(I^0)$.

{(ii)} If $I$ denotes and ideal of $\O_n$ of finite colength generated by monomials, then the equivalence between the conditions $\lct(I)=\DP(I)$ and $I$ is diagonal
also follows as a corollary of a more general result stated for multi-circled plurisubharmonic singularities and proved by
Rashkovskii in  \cite[Theorem 1.5]{RashkCRASP} following techniques from pluripotential theory.
\end{rem}

\begin{ex}\label{contraex}
Let us consider the polynomials of $\O_2$ given by $g_1=(x+y)^2+y^4$ and $g_2=(x+y)y^2$.
Let $I$ be the ideal of $\O_2$ generated by $g_1$ and $g_2$.
Then $e_1(I)=\ord(I)=2$ and $e(I)=8$.
If we apply to $I$ the linear coordinate change
$(x,y)\mapsto (x-y,y)$, then we obtain the ideal $J=\langle x^2+y^4, xy^2\rangle$.
We observe that $J$ is a Newton non-degenerate ideal (see \cite{CBAJoint} or \cite{Tessiermbp}),
which implies that $\overline J=J^0=\overline{\langle x^2, y^4\rangle}$. Then
$J$ is diagonal and hence $\lct(I)=\lct(J)=\lct(J^0)=\frac{3}{4}=\frac 12+\frac 28=\DP(I)$.

We observe that $\Gamma_+(I)$ has a unique compact face $\Delta$ of dimension $1$, hence $I$ is diagonal if and only if
$\overline I$ is generated by monomials, which is to say that $I$ is Newton non-degenerate.
Following the notation introduced in \cite[p.\,398]{CBAJoint} we see that $(g_1)_\Delta=(x+y)^2$, $(g_2)_\Delta=0$,
and hence the solutions of the system $(g_1)_\Delta=(g_2)_\Delta=0$ are not contained in $\{(x,y)\in\C^2:xy=0\}$.
Then $I$ is not Newton non-degenerate, by \cite[Proposition 3.6]{CBAJoint} and thus $I$ is not a diagonal ideal, although $\lct(I)=\DP(I)$.
\end{ex}

\vspace{0.5cm}

\noindent{\bf Acknowledgement.} This work has been supported by DGICYT Grant MTM2015-64013-P.
The author wishes to thank Professor A. Rashkovskii for his helpful comments.

%%%%%%%%%%%%%%%%%%%%%%%%%%%%%%%%%%%%%%%%%%%%%%%%%%%%%%%%%%%%%%%%%%%%%%%%%%%%%%%%%%%%%%%%%%%%%%%%%%%%%%%%%%%%%%%%%%
%\newpage
%%%%%%%%%%%%%%%%%%%%%%%%%%%%%%%%%%%%%%%%%%%%%%%%%%%%%%%%%%%%%%%%%%%%%%%%%%%%%%%%%%%%%%%%%%%%%%%%%%%%%%%%%%%%%%%%%%

\end{document}